\documentclass[reqno,11pt, a4] {amsart}
\usepackage[active]{srcltx}
\usepackage[usenames,dvipsnames,svgnames,table]{xcolor}
\usepackage[poly,curve,line, frame, all, arc]{xy}
\usepackage{pb-diagram}
\usepackage{epstopdf}
\usepackage{ulem}

\usepackage{mathrsfs}
\usepackage{amsmath}
\usepackage{amssymb}
\usepackage{amscd}
\usepackage{amsthm}
\usepackage[latin1]{inputenc}
\usepackage{graphics}
\usepackage{graphicx}
\usepackage{varioref}

\labelformat{enumi}{(#1)}

\newtheorem{teo}{Theorem}[section]

\newtheorem{remark}{Remark}[section]
\newtheorem{prop}{Proposition}[section]

\newtheorem{lemma}{Lemma}[section]

\newcounter{yuppo}
\newtheorem{yuppi}{Theorem}[yuppo]

\newcommand{\cds}{\cdots}
\newcommand{\cd}{\cdot}

\renewcommand{\phi}{\varphi}

\newcommand{\ra}{\rightarrow}
\newcommand{\lra}{\longrightarrow}
\newcommand{\C}{\mathbb{C}}
\newcommand{\R}{\mathbb{R}}

\newcommand{\SU} {\operatorname{SU}}
\newcommand{\su} {\mathfrak{su}}
\newcommand{\Sl}{\operatorname{SL}}
\newcommand{\Gl}{\operatorname{GL}}

\newcommand{\Ad}{\operatorname{Ad}}

\newcommand{\meno}{^{-1}}
\newcommand{\Zeta}{{\mathbb{Z}}}

\newcommand{\liu}{\mathfrak{u}}

\newcommand{\lia}{\mathfrak{a}}

\newcommand{\liek}{\mathfrak{k}}
\newcommand{\lier}{\mathfrak{r}}

\newcommand{\lieg}{\mathfrak{g}}

\newcommand{\lieb}{\mathfrak{b}}
\newcommand{\liep}{\mathfrak{p}}
\newcommand{\lieq}{\mathfrak{q}}
\newcommand{\lied}{\mathfrak{d}}
\newcommand{\liez}{\mathfrak{z}}
\newcommand{\lies}{\mathfrak{s}}

\newcommand{\vacuo}{\emptyset}
\newcommand{\OO}{\mathcal{O}}               
\renewcommand{\c}{{\widehat{\OO}}}          
\newcommand{\polp}{{P}}                     
\newcommand{\ext}{\operatorname{ext}}       
\newcommand{\sx}{\langle}                   
\newcommand{\xs}{\rangle}
\newcommand{\relint}{\operatorname{relint}} 

\newcommand{\faces}{\mathscr{F}(E)}       
\newcommand{\facesp}{\mathscr{F}(\polp)}    

\newcommand{\CF}{C_F}

\newcommand{\scalo}{\sx \, , \, \xs}

\newcommand{\noparty}[1]{}

\newcommand{\xmax}{X_{\max}^\beta}

\newcommand{\mup}{\mu_\liep}
\newcommand{\mua}{\mu_\lia}
\newcommand{\mupb}{\mu_\liep^\beta}
\newcommand{\nora}{\mathcal N_{K} (\lia)}
\newcommand{\Dir}{\mathrm{Dir}}
\newcommand{\cc}{E}
\newcommand{\perpp}{\perp}

\begin{document}

\title{Invariant convex sets in polar representations}
\author{Leonardo Biliotti} \author{Alessandro Ghigi} \author{Peter
  Heinzner}
\subjclass[2000]{22E46; 
  53D20 
}
\begin{abstract}
  We study a compact invariant convex set $E$ in a polar
  representation of a compact Lie group. Polar rapresentations are
  given by the adjoint action of $K$ on $\liep$, where $K$ is a
  maximal compact subgroup of a real semisimple Lie group $G$ with Lie
  algebra $\lieg = \liek \oplus \liep$. If $\lia \subset \liep$ is a
  maximal abelian subalgebra, then $P=E\cap \lia$ is a convex set in
  $\lia$.  We prove that up to conjugacy the face structure of $E$ is
  completely determined by that of $P$ and that a face of $E$ is
  exposed if and only if the corresponding face of $P$ is exposed.  We
  apply these results to the convex hull of the image of a restricted
  momentum map.
\end{abstract}
\address{Universit\`{a} di Parma} \email{leonardo.biliotti@unipr.it}
\address{Universit\`a di Milano Bicocca}
\email{alessandro.ghigi@unimib.it} \address{Ruhr Universit\"at Bochum}
\email{peter.heinzner@rub.de} \thanks{The first two authors were
  partially supported by FIRB 2012 ``Geometria differenziale e teoria
  geometrica delle funzioni'', by a grant of the Max-Planck Institut
  f\"ur Mathematik, Bonn and by GNSAGA of INdAM. The second author was
  also supported by PRIN 2009 MIUR ``Moduli, strutture geometriche e
  loro applicazioni''. The third author was partially supported by
  DFG-priority program SPP 1388 (Darstellungstheorie)}
\maketitle

\normalem
The boundary of a compact convex set is the union of its faces. Among
the faces, the simplest ones are the exposed ones. They are given by
the intersection of the convex set with a supporting hyperplane.  In
\cite{biliotti-ghigi-heinzner-1-preprint, biliotti-ghigi-heinzner-2}
we studied the convex hull $\widehat{\mathcal O}$ of a $K$-orbit $\OO$
in $\liep$, where $\liep$ is given by the Cartan decomposition $\lieg
=\liek \oplus \liep$ of a reductive Lie algebra $\mathfrak g$ and $K$
acts on $\liep$ by the adjoint representation.
%
%
%
In this paper we use the results of \cite{biliotti-ghigi-heinzner-2}
and show that a substantial part of them holds for any $K$--invariant
compact convex set $E$ of $\liep$.  More precisely we study the faces
of $E$.
We show in Proposition \ref{fgo} that for a face $F$ of $E$ there
exists a subalgebra $\lies\subset \liep$ such that $F$ is a subset of
$\liep^{\lies}=\{x\in \liep:\, [x,\lies]=0\}$ and $F$ is invariant
with respect to the action of $K^{\lies}=\{h\in K:\,
\mathrm{Ad}(h)(\lies)=\lies\}$, where $\mathrm{Ad}$ denotes the
adjoint representation.

If we fix a maximal abelian subalgebra $\lia \subset \liep$, then the
set $P=E \cap \lia$ is convex and invariant with respect to the action
of the normalizer $\mathcal{N}_{K} (\lia)=\{h\in K:\,
\mathrm{Ad}(h)(\lia)=\lia\}$ of $\lia$ in $K$.  The $\mathcal{N}_K
(\lia)$--action on $P$ induces an action on the set of faces of $P$.
Similarly $K$ acts on the set of faces of $E$. Denote these sets by
$\mathscr{F}(P)$ respectively by $\mathscr{F}(E)$.  If $\sigma $ is a
face of $P$, let $\sigma^\perp$ denote the orthogonal complement in
$\lia$ of the affine hull of $\sigma$ (see Section \ref{uno}).  Our
main result is
\begin{teo}
  \label{main1}
  The map $\facesp \ra \faces$, $\sigma \mapsto K^{\sigma^\perp}\cd
  \sigma$ is well--defined and
  induces a bijection between $\facesp/\nora$ and $\faces/K$.
\end{teo}
An application of Theorem \ref{main1} is the following result.
\begin{teo}\label{main2}
  The faces of $E$ are exposed if and only if the faces of $P$ are
  exposed.
\end{teo}
Interesting $K$--invariant compact subsets of $\liep$ often arise as
images of restricted momentum or gradient mappings.  More precisely,
let $U$ be a compact connected Lie group which acts by biholomorphism
and in a Hamiltonian fashion on a compact K\"ahler manifold $Z$ with
momentum map $\mu:Z \lra \liu$.
Let $G\subset U^\C$ be a connected Lie subgroup of $U^\C$ which is
\emph{compatible} with respect to the Cartan decomposition of
$U^\C$. This means that $G$ is a closed subgroup of $U^\C$ such that
$G=K\exp (\liep)$, where $K=U\cap G$ and $\liep=\lieg \cap i\liu$
\cite{heinzner-schwarz-stoetzel,heinzner-stoetzel-global}. Let
$X\subset Z$ be a $G$-invariant compact subset of $Z$. We have the
restricted momentum map or the gradient map $\mup:X \lra \liep$ in the
sense of \cite{heinzner-schwarz-stoetzel} (see also Section
\ref{gradient-map}) and we denote by $E=\widehat{\mup(X)}$ the convex
hull of the $K$-invariant set $\mup(X)$.  If $\lia$ is a maximal
abelian subalgebra of $\liep$ and $\pi$ is the orthogonal projection
onto $\lia$, then $\mua=\pi \circ \mup:X \lra \lia$ is the gradient
map with respect to $A=\exp (\lia)$.  Since $P=\cc\cap
\lia=\widehat{\mu_{\lia} (X)}$ is a convex polytope (Proposition
\ref{momentum-polytope}), we deduce the following.
\begin{teo}\label{momentum-polytope}
  All faces of $\widehat{\mup (X)}$ are exposed.
\end{teo}
A reformulation of Theorem \ref{momentum-polytope} is that the faces
of $\cc$ correspond to maxima of components of the gradient map. This
observation will be used to realize a close connection between the
faces of $E$ and parabolic subgroups of $G$.  More precisely, for any
face $F\subset \cc$ let $X_F := \mup\meno(F)$ and let $Q^F = \{g \in
G: g \cd X_F =X_F\}$. Then $X_F$ is the set of maximum points of an
appropriately chosen component of the gradient map and $Q^F$ is a
parabolic subgroup of $G$.

If $X$ is a $G$-stable compact submanifold of $Z$, then for any face
$F$, one can construct an open neighbourhood $X_F^{-}$ of $X_F$ in
$X$, which is an analogue of an open Bruhat cell.  Moreover there is a
smooth deformation retraction of $X_F^{-} $ onto $ X_F$.  See Theorem
\ref{urg} for more details.

\medskip

{\bfseries \noindent{Acknowledgements.}}  The first two authors are
grateful to the Fakul\-t\"at f\"ur Mathematik of Ruhr-Universit\"at
Bochum for the wonderful hospitality during several visits.  They also
wish to thank the Max-Planck Institut f\"ur Mathematik, Bonn for
excellent conditions provided during their visit at this institution,
where part of this paper was written.

\section{Group theoretical description of the faces}
\label{uno}
We start by recalling the basic definitions and results regarding
convex bodies. For more details see e.g.
\cite{schneider-convex-bodies}.  Let $V$ be a real vector space with
scalar product $\langle \cdot, \cdot \rangle$.  A \emph{convex body}
$E \subset V$ is a convex compact subset of $V$.  Let
$\mathrm{Aff}(E)$ denote the affine span of $E$.  The interior of $E$
in $\operatorname{Aff}(E)$ is called the \emph{relative interior} of
$E$ and is denoted by $\relint E$. By definition a \emph{face} of $E$
is a convex subset $F\subset E$ such that $x,y\in E$ and
$\relint[x,y]\cap F\neq \vacuo$ implies $[x,y]\subset F$. A face
distinct from $E$ and $\emptyset$ is called a \emph{proper face}.  The
\emph{extreme points} of $E$ are the points $x\in E$ such that $\{x\}$
is a face. We will denote by $\ext E$ the set of the extreme points of
$E$. The set $\ext E$ completely determines the convex body $E$ since
the convex hull of $\ext E$ coincides with $E$ and it is the smallest
subset of $E$ with this property.  If $F$ is a face of $E$, we denote
by $\operatorname{Dir}(F)$ the vector subspace of $V$ defined by
$\mathrm{Aff}(F)$, i.e.  $\mathrm{Aff}(F)=p+\mathrm{Dir}(F)$.  We call
$\operatorname{Dir}(F)$ the \emph{direction} of $F$.  Every vector
$\beta \in V$ defines an \emph{exposed face} $F=F_\beta (E)=\{x\in
E:\, \langle x, \beta \rangle= \max_{y\in E} \langle y , \beta \rangle
\}$ with $\mathrm{Dir}(F_\beta (E))\subset \{\beta \}^{\perpp}$. In
general not all faces of a convex set are exposed, see Fig. \ref{figP}
for an example.  For any exposed face $F$ the set
\begin{gather}
  \label{def-CF}
  C_F=\{\beta \in V:\, F=F_\beta (E)\},
\end{gather}
is a convex cone.
The faces of $E$ are closed. If $F_1$ and $F_2$ are faces of $E$ and
they are distinct, then $\relint F_1 \cap \,\relint F_2
=\emptyset$. Moreover the convex body $E$ is the disjoint union of the
relative interiors of its faces (see
{\cite[p. 62]{schneider-convex-bodies}}).

We are interested in invariant convex bodies in polar
representations. A theorem of Dadok \cite{dadok-polar} asserts that we
can restrict ourselves to the following setting.

Let $\lieg$ be a semisimple Lie algebra with a Cartan involution
$\theta$ and let $B$ be the Killing form of $\lieg$. Then $\lieg=\liek
\oplus \liep$, is the eigenspace decomposition of $\lieg$ in $1$ and
$-1$ eigenspaces of $\theta$ and they are orthogonal under
$B$. Moreover, $B$ restricted to $\liek$, respectively $\liep$, is
negative definite, respectively positive definite. In the sequel we
denote $\langle \cdot,\cdot \rangle =B_{|_{\liep \times \liep}}$ which
is a $K$-invariant scalar product.  Out object of study will be a
$K$-stable convex body $E\subset \liep$. For for any $A,B\subset
\liep$ we set \begin{gather*}
  A^B :=     \{\eta\in A: [\eta, \xi ] =0, \mathrm{for\ all}\ \xi \in B\} \\
  G^B := \{g\in G: \Ad g (\xi ) = \xi,  \mathrm{for\ all}\ \xi \in B\}, \\
  K^B : =K\cap G^B.
\end{gather*}
where $\mathrm{Ad}$ denotes the adjoint representation. In the sequel
we denote by $k\cdot x=\mathrm{Ad}(k)(x)$ the action of $K$ on $\liep$
by linear isometries.

Faces of $K$--invariant convex bodies in $\liep$ are closely connected
to orbits of subgroups of $K$ which are given as centralizers. More
precisely for any nonzero $\beta$ in $\liep$ we have the Cartan
decomposition $\lieg^\beta=\liek^\beta \oplus \liep^\beta$ of the Lie
algebra of the centralizer $G^\beta$ of $\beta$ in $G$.
\begin{prop}\label{exposed-face}
  Let $F=F_\beta (E)$ be an exposed face of $E$. Then
  \begin{enumerate}
  \item $F\subset \liep^{\beta}$ and $F$ is $K^\beta$--stable;
  \item $\mathrm{Dir}(F)\subset \beta^{\perp}$, where $\perp$ is in
    $\liep$.
  \end{enumerate}
\end{prop}
\begin{proof}
  If $x\in F_\beta (E)$, then $\widehat{K\cdot x} \subset E$ since $E$
  is $K$-invariant. Moreover, we have
  \begin{gather*}
    \max_{y\in E} \sx y, \beta \xs = \max_{y\in \widehat{K\cdot x} }
    \sx y, \beta\xs=\sx x,\beta \xs.
  \end{gather*}
  Corollary 3.1 in \cite{biliotti-ghigi-heinzner-2} implies $F_\beta
  (\widehat{K\cdot x})\subset \liep^{\beta}$. Therefore $x \in
  \liep^\beta$.  This proves a). Part b) follows since $F $ is
  contained in an affine hyperplane orthogonal to $\beta$.
\end{proof}
For an arbitrary face of $E$ we have the following.
\begin{prop}\label{fgo}
  Let $F\subset E$ be a face. Then there exists an abelian subalgebra
  $\lies \subset \liep$ such that
  \begin{enumerate}
  \item $F\subset \liep^{\lies}$ and $F$ is $K^{\lies}$--stable;
  \item $\mathrm{Dir}(F)\subset \lies^{\perpp}$;
  \end{enumerate}
\end{prop}
\begin{proof}
  We may fix a maximal chain of faces $ F=F_0 \subsetneq F_1
  \subsetneq \cds \subsetneq F_k=E$ (see \cite[Lemma
  2]{biliotti-ghigi-heinzner-1-preprint}). If $k=0$, then $F=E$ and
  $\lies=\{0\}$. Assume the theorem is true for a face contained in a
  maximal chain of length $k$.  Then the claim is true for $F_1$ and
  consequently there exists $\lies_1 \subset \liep$ such that $F_1
  \subset \liep^{\mathfrak s_1}$, $F_1$ is $K^{\lies_1}$-stable and
  $\mathrm{Dir}(F_1 )\subset \lies_1^{\perpp}$. $F$ is an exposed face
  of $F_1$. Let $\beta'\in \liep^{\lies_1}$ such that
  $F=F_{\beta'}(F_1)$ and set $\lies:=\R \beta' \oplus \lies_1$. Then
  $F\subset \liep^{\lies}$, $F$ is
  $(K^{\lies_1})^{\beta'}=K^{\lies}$--stable and
  $\mathrm{Dir}(F)\subset \lies^{\perpp}$.
\end{proof}
Let $\lia\subset \liep$ be a maximal abelian subalgebra of $\liep$ and
let $\pi:\liep \lra \lia$ be the orthogonal projection onto
$\lia$. Then $P=E\cap \lia$ is a convex subset of $\lia$ which is
$\mathcal N_{K}(\lia )$--stable. The proof of the following Lemma is
given in \cite{gichev-polar}.
\begin{lemma}
  \label{gichev}
  (i) If $E \subset \liep$ is a $K$--invariant convex subset, then
  $E\cap \lia = \pi (E)$ and $K\cd \pi(E) =E$. (ii) If $C \subset
  \lia$ is a $\mathcal N_{K} (\lia)$-invariant convex subset, then $K
  \cdot C $ is convex and $\pi (K\cdot C )= C$.
\end{lemma}
\begin{lemma}\label{closed}
  Let $U$ be a compact Lie group and let $\lieg\subset \liu^\C$ be a
  semisimple $\theta$-invariant subalgebra. Then any Lie subgroup with
  finitely many connected components and with Lie algebra $\lieg$ is
  closed and compatible.
\end{lemma}
\begin{proof}
  We fix an embedding $U \hookrightarrow\operatorname{U}(n)$ such that
  the Cartan involution $X \mapsto (X\meno)^*$ of $\Gl(n, \C)$
  restricts to $\theta$.  Then $G$ is closed in $\Gl(n, \C)$ (see
  \cite[p. 440]{knapp-beyond} for a proof) and hence also in $U^\C$.
  Since $\lieg$ is $\theta$-invariant, also $G$ is, and $\theta$
  restricts to the Cartan involution of $G$. This shows that $G$ is
  compatible.
\end{proof}
If $G \subset U^\C$ is compatible with Lie algebra $\lieg=\liek \oplus
\liep$, then $\lieg$ is real reductive and there is a nondegenerate
$K$--invariant bilinear form $B:\lieg \times \lieg \lra \R$ which is
positive definite on $\liep$, negative definite on $\liek$ and such
that $B(\liek,\liep)=0$. Indeed, fix a $U$-invariant inner product
$\scalo$ on $\liu$. Let $\scalo$ denote also the inner product on
$i\liu$ such that multiplication by $i$ be an isometry of $\liu$ onto
$i\liu$.  Define $B$ on $\liu^{\C}$ imposing $B(\liu, i \liu)=0$, $B=
-\scalo$ on $\liu$ and $B= \scalo$ on $i\liu$.  Therefore $B$ is
$\mathrm{Ad}\, U^\C$--invariant and non-degenerate and its restriction
to $\lieg$ satisfies the above conditions.

Let $\mathfrak q$ be a $K$--invariant subspace of $\liep$. Then
$[\mathfrak q,\mathfrak q]$ is a $K$--invariant linear subspace of
$\liek$ and therefore an ideal of $\liek$. Since $K$ is compact, we
have the following $K$--invariant splitting $\liek=[\mathfrak
q,\mathfrak q]\oplus \liek'.$ In particular $\liek'$ is an ideal of
$\liek$ commuting with $[\mathfrak q,\mathfrak q]$. Let $\liep
=\mathfrak q \oplus \mathfrak q'$ be a $K$--invariant splitting of
$\liep$. Since
\[
B( [\mathfrak q, \mathfrak q'],\liek ) = B( \mathfrak q , [\mathfrak
k,\mathfrak q'] ) \subset B( \mathfrak q, \mathfrak q' ) =0,
\]
this shows that $[\mathfrak q,\mathfrak q']=0$ and so $[\mathfrak q',
[\mathfrak q , \mathfrak q]] = [\mathfrak q, [\mathfrak q,\mathfrak
q']]=0$.  Moreover $\liep=\mathfrak q \oplus \mathfrak q'$ implies
that $\mathfrak h =[\mathfrak q,\mathfrak q] \oplus \mathfrak q$ and
$\mathfrak h'=\liek'\oplus \mathfrak q'$ are compatible $K$--invariant
commuting ideal of $\lieg$.

If a $K$--invariant linear subspace $\lieq \subset \liep$ is fixed,
one gets decomposition of $\lieg$, and so of $G$. This is
decomposition is the content of the next Proposition. We will need it
in the case where $F\subset \liep$ is a $K$--invariant convex body and
$\lieq$ is such that $\operatorname{Aff}(F) =x_0 +\mathfrak q$.
\begin{prop}\label{deco-face}
  Let $G \subset U^\C$ be a compatible subgroup with Lie algebra
  $\lieg = \liek \oplus \liep$ and let $\mathfrak q \subset \liep$ be
  a linear $K$--invariant subspace. Let $\lieg =\mathfrak h \oplus
  \mathfrak h'$ where $\mathfrak h=[\mathfrak q,\mathfrak q] \oplus
  \mathfrak q$ and $\mathfrak h'=\mathfrak h^{\perp_B}$. Then the
  following hold.
  \begin{enumerate}
  \item $\mathfrak h$ and $\mathfrak h'$ are compatible $K$--invariant
    commuting ideal of $\lieg$;
  \item Let $K_1$ be the connected Lie subgroup of $G$ with Lie
    algebra $ \liek \cap [\mathfrak h,\mathfrak h]$. Then $K_1 \exp
    (\mathfrak q)$ is a connected compatible subgroup of $G$ and any
    two maximal subalgebras of $\mathfrak q$ are congiugate by an
    element of $K_1$.
  \item Let $K_2$ be the connected Lie subgroup of $G$ with Lie
    algebra $ \liek \cap [\mathfrak h' , \mathfrak h']$.  Then any two
    maximal subalgebras of $\mathfrak q'$ are congiugate by an element
    of $K_2$. 
  \end{enumerate}
\end{prop}
\begin{proof}
  We have proved (a) in the above discussion. Let $\lieb : =[\mathfrak
  h, \mathfrak h]$. Then $\mathfrak h = \liez(\mathfrak h) \oplus
  \lieb$ and $\lieb$ is semisimple.  Denote by $B$ the connected
  subgroup of $U^\C$ with Lie algebra $\lieb$.  By Lemma \ref{closed}
  $B$ is a closed subgroup of $U^\C$. Set $\liez_{\liep} :=
  \liez(\mathfrak h) \cap \liep$ and $\lied: = \lieb \oplus
  \lia$. Then $\lied$ is a reductive Lie algebra and $\exp \lia $ is a
  compatible abelian subgroup commuting with $B$. Thus $D:=B \cd \exp
  \lia$ is a connected closed subgroup with Lie algebra
  $\lied$. Moreover $D \cap U = B\cap U$ and $\exp (\lieb \cap \liep )
  \cd \exp \lia = \exp (\lieb \cap \liep \oplus \lia) = \exp (\lied
  \cap \liep)$. This shows that $D$ is compatible. Since $D\cap U$
  coincides with $K_1$ and $D$ is connected the last statement in (b)
  follows from standard properties of compatible subgroups (see e.g.
  Prop. 7.29 in \cite{knapp-beyond}; note that a connected compatible
  subgroup is a reductive group in the sense of
  \cite[p. 446]{knapp-beyond}).  This proves (b).  For (c) the same
  argument applies more directly. It is enough to observe that the
  connected Lie subgroup $H'' \subset G$ with Lie algebra $[\mathfrak
  h',\mathfrak h']$ is semisimple, compatible and connected and that
  $K_2 = H''\cap U$.
\end{proof}
\begin{remark}\label{rem1}
  The compatible subgroup $G$ in the previous Proposition is not
  assumed to be connected. Nevertheless the constructions in (b) and
  (c) depend only on $G^0$.  Thus considering $G^0$ in place of $G$
  makes no difference.
\end{remark}
\begin{lemma}\label{spezza-abeliana}
  Let $\lieg = \lieg_1 \oplus \lieg_2$ be a reductive Lie algebra and
  $\lieg_i$ ideals. If $\lia \subset \liep$ is a maximal subalgebra,
  then $\lia_i := \lia \cap \liep_i $ is a maximal subalgebra of
  $\liep_i$ and $\lia = \lia_1 \oplus \lia_2$.
\end{lemma}
If $\sigma $ is a face of $P$, let $\sigma^\perp$ denote the
orthogonal (inside $\lia$) to the direction of the affine hull of
$\sigma$.
\begin{lemma}\label{faccia-faccetta}
  Let $F$ be a face and let $\lies$ be as in Proposition \ref{fgo}.
  Let $\lia \subset \liep$ be a maximal abelian subalgebra containing
  $\mathfrak s$.  Set $\sigma : = \pi(F)$. Then $\sigma$ is a face of
  $P$, $\sigma=F\cap \lia$ and $F=K^{\sigma^{\perp} } \cdot
  \sigma$. Moreover $F$ is a proper face if and only if $F\cap \lia$
  is.
\end{lemma}
\begin{proof}
  By Proposition \ref{fgo} $F\subset \liep^\lies$ is a
  $K^\lies$--stable convex set. By Lemma \ref{gichev} we get
  $\sigma=\pi(F)=F\cap \lia$ and this is a face $P$ by \cite[Lemma
  11]{biliotti-ghigi-heinzner-1-preprint}.  Since $\mathrm{Dir}(F)$ is
  contained in the orthogonal complement of $\lies$, and $\Dir(\sigma
  ) \subset \Dir(F)$, we have $\Dir(\sigma ) \subset \lia \cap
  \lies^\perp$.  Then $\sigma ^\perp \subset \lies$.  Hence
  $K^{\sigma^{\perp}} \cdot \sigma \subset K^{\lies} \cd \sigma
  \subset F$. We prove the reverse inclusion.  If $y\in F$, then
  $F\cap \widehat{K\cdot y}$ is a face of $\widehat{K\cdot y}$.  Set
  $\tilde \sigma = \pi (F\cap \widehat{K\cdot y})$. We have $\tilde
  \sigma \subset \sigma$ and by Proposition 3.6 in
  \cite{biliotti-ghigi-heinzner-2} we also have that $F\cap
  \widehat{K\cdot y}=K^{\tilde \sigma^\perp} \cdot \tilde \sigma$. On
  the other hand, $\sigma^\perp \subset \tilde \sigma^\perp$, so
  $K^{\tilde \sigma^\perp}\subset K^{\sigma^\perp}$ and
  \[
  F\cap \widehat{K\cdot y} =K^{\tilde \sigma^\perp} \cdot \tilde
  \sigma \subset K^{\sigma^\perp} \cdot \sigma.
  \]
  This implies $F= K^{\sigma^\perp} \cd \sigma$. Note that $F$ is
  proper if $\sigma$ is. It remains to prove that $\sigma$ is proper,
  when $F$ is proper.

  Let $\mathrm{Aff}(E)=x_o +\mathfrak q_E$. Note that $\mathfrak
  q_E=\{x-y:\, x,y\in \mathrm{Aff}(E)\}$ implies that $\mathfrak q_E$
  is $K$--invariant.  Since $K$ acts on $\liep$ by isometries, we may
  assume that $x_o$ is orthogonal to $\mathfrak q$. Note that $x_o$ is
  uniquely defined by this condition.  It follows that $x_o$ is a $K$
  fixed point and $E=x_0 +E_1$, where $E_1$ is a $K$--invariant convex
  body of $\mathfrak q_E$.  Proposition \ref{deco-face} applied to
  $\mathfrak q_E$ yields $K_1, K_2$ such that $G_1 =K_1 \exp
  (\mathfrak q_E)$ is a connected compatible semisimple real Lie
  group, $K=K_1 \cdot K_2$ and for any $x\in E$ we have
  \[
  K \cdot x= K \cdot (x_o +x_1)=x_o + K \cdot x_1=x_o + K_1 \cdot
  x_1=K_1 \cdot x.
  \]
  since $\mathfrak q_E$ is fixed pointwise by $K_2$. By Lemma
  \ref{spezza-abeliana}, $\lia=\lia_E \oplus \lia_E'$, where $\lia_E$
  is a maximal abelian subalgebra of $\mathfrak q_E$ and $\lia_E'$ is
  a maximal abelian subalgebra of $\mathfrak q_E'$. Since
  $\pi(E)=\pi(x_o)+ \pi(E_1)$ and $\mathrm{Dir} (E_1)=\mathfrak q_E$,
  it follows that the direction of $\pi(E)$ is $\lia_E$. If
  $\sigma=\pi(F)=\pi(E)=E\cap \lia$, then $\sigma^\perp=\lia_E'$ and
  so $K_1 \subset K^{\lia_E'}$. It follows that
  \[
  F= K^{\lia_E'} \cdot (E\cap \lia )= K_1 \cdot (E\cap \lia)= K \cdot
  (E\cap \lia)=E.
  \]
  where the last equality follows by Lemma \ref{gichev}. Hence, if $F$
  is proper, then $\sigma=\pi(F) \subsetneq \pi(E)=E \cap \lia$.
\end{proof}
\begin{prop}\label{faccia-faccetta2}
  Let $F$ be a proper face and let $\lies$ as in Proposition
  \ref{fgo}.  Let $\lia \subset \liep$ be a maximal abelian subalgebra
  containing $\mathfrak s$. Then $F$ is exposed if and only if $F\cap
  \lia$ is.
\end{prop}
\begin{proof}
  Assume that there exists $\beta\in \liep$ such that $F=F_\beta
  (E)$. Since $F\cap \lia=\sigma$ is a proper face of $P$, the point
  $\beta$ is not orthogonal to $\lia$. We have $\beta=\beta_1 \oplus
  \beta_2$, with $\beta_1 \in \lia$ different from zero and $\beta_2$
  orthogonal to $\lia$. Therefore $F_\beta (E) \cap \lia=F_{\beta_1}
  (E) \cap \lia=F_{\beta_1}(P)=\sigma$.  Now, assume that there exists
  $\beta\in \lia$ such that $\sigma = F_\beta (\polp)$.  Let $F' : =
  F_\beta( E)$.  By Proposition \ref{exposed-face} $F' \subset
  \liep^\beta$. Moreover $\lia \subset \liep^\beta$ since $\beta\in
  \lia$.  By Lemma \ref{faccia-faccetta} the intersection of a face
  with $ \lia$ determines the face. Since $F' \cap \lia=F_\beta
  (P)=\sigma = F\cap \lia$ we conclude that $F=F'$. Thus $F$ is
  exposed.
\end{proof}
\begin{figure}
  \centering
  \includegraphics[scale=0.3] {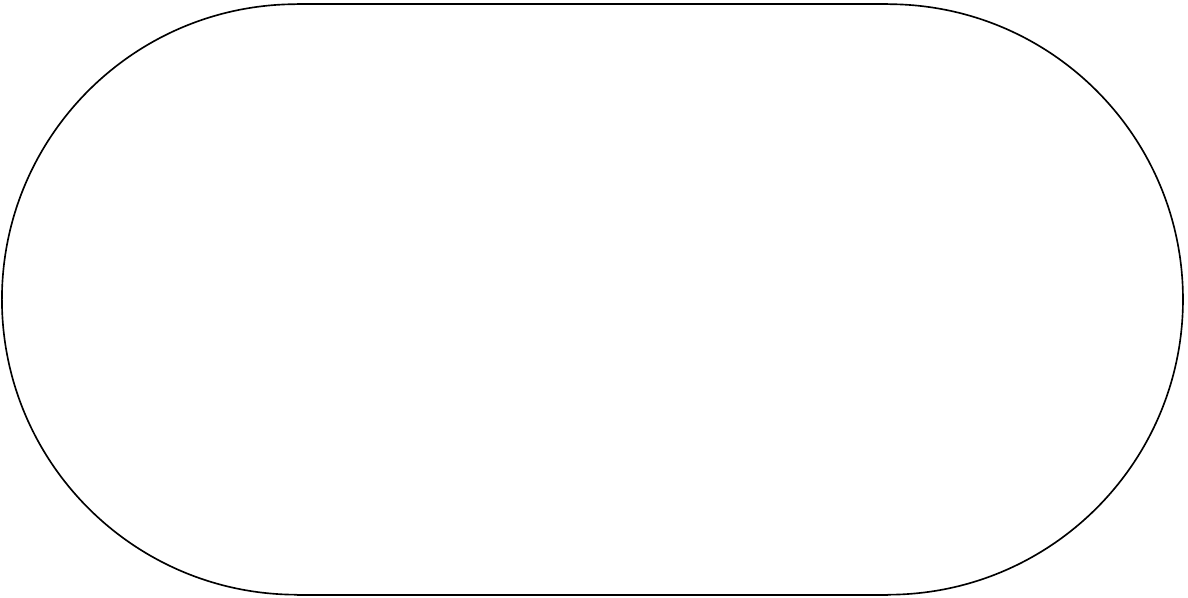}
  \caption{}
  \label{figP}
\end{figure}
\begin{remark}
  Given a Weyl--invariant convex body $P \subset \lia $, set $E:=K\cd
  P$.  By Lemma \ref{gichev} $E$ is a $K$-invariant convex body in
  $\liep$ and $P =E\cap \lia$.  Thus a general $P$ can be realized as
  $E\cap \lia$.  A general Weyl--invariant convex body $P$ can have
  non--exposed faces. For example take $G=U^\C = \Sl(2,\C) \times
  \Sl(2,\C)$ and $K=\SU(2) \times \SU(2)$. Then $\lia = \R^2$ and the
  Weyl group is isomorphic to $\Zeta/2 \times \Zeta/2$ where the
  generators are given by the reflections on the axes.  The picture in
  Fig. \ref{figP} is a Weyl--invariant $P$ with exactly $4$
  non--exposed faces. By the Proposition also the corresponding body
  $E \subset i\su(2) \oplus i\su(2)$ has non--exposed faces.
\end{remark}
\section{Proof of the main results}
Let $\lia\subset \liep$ and define the following map
\[
\Upsilon: \facesp \lra\faces, \ \ \sigma \mapsto K^{\sigma^\perp}\cdot
\sigma
\]
Since $\sigma$ is $\mathcal N_{K^{\sigma^\perp}} (\lia)$-invariant, it
follows from Lemma \ref{gichev} that $\Upsilon(\sigma)$ is a face of
$E$.  \setcounter{yuppo}{0}
\begin{yuppi}
  The map $\Upsilon$ induces a bijection between $\facesp/\mathcal
  N_K(a)$ and $\faces/K$.
\end{yuppi}
\begin{proof}
  Set $\mathcal N:=\nora$.  We first show that $\Upsilon$ is $\mathcal
  N$-equivariant. Let $w\in \mathcal N$. Then $\sigma'= w \sigma$
  implies $K^{{\sigma'}^\perp}=w K^{\sigma^\perp} w^{-1}$ and
  therefore $\Upsilon(\sigma')= w \Upsilon (\sigma)$. This means that
  the map
  \[
  \tilde \Upsilon: \facesp / \mathcal N \lra \faces /K, \ \ [\sigma]
  \mapsto K^{\sigma^\perp}\cdot \sigma
  \]
  is well-defined. Next, we prove that $\tilde \Upsilon$ is
  injective. Assume for some $g\in K$ $g\cdot F_1=F_2$ where
  $F_1=\Upsilon (\sigma_1)$ and $F_2=\Upsilon (\sigma_2)$. Since $F_2
  =K^{\sigma_2^\perp}\cdot \sigma_2$, the face $F_2$ is a
  $K^{\sigma_2^\perp}$--invariant convex body.  Moreover $\sigma_2
  \subset \lia \subset \liep^{\sigma_2^\perp}$ and
  $\liep^{\sigma_2^\perp}$ is
  $K^{\sigma_2^\perp}$--invariant. Therefore $F_2$ is contained in
  $\liep^{\sigma_2^\perp}$. It follows that $\mathrm{Aff}(F_2)=x_o +
  \mathfrak q_{F_2}$, where $\mathfrak q_{F_2}$ is a
  $K^{\sigma_2^\perp}$ invariant subspace of $\liep^{\sigma_2^\perp}$,
  $x_o$ is a fixed $K^{\sigma_2^\perp}$ point and it is orthogonal
  orthogonal to $\mathfrak q_{F_2}$. We apply Proposition
  \ref{deco-face} to the group $G^{\sigma_2^{\perp}}$ and $\mathfrak
  q_{F_2}$. Thus $\mathfrak h_{F_2} = [\mathfrak q_{F_2}, \mathfrak
  q_{F_2}] \oplus \mathfrak q_{F_2}$ and its orthogonal complement in
  $\lieg^{\sigma_2^\perp}$, that we denote by $\mathfrak h_{F_2}'$,
  are commuting ideal. The Proposition \ref{deco-face} also yields
  subgroups $K_1, K_2 \subset K^{\sigma_2\perp}$ such that any two
  maximal subalgebras in $\mathfrak q_{F_2}$, respectively $\mathfrak
  q_{F_2}'$, are interchanged by $K_1$, respectively $K_2$.  Since
  $\sigma_2 \subset \lia$, also $\Dir(\sigma_2) \subset \lia$ and we
  may decompose $\lia =\Dir(\sigma_2) \oplus \sigma_2^\perp$. But
  $\Dir(\sigma_2)$ is contained also in $\lieq_{F_2}$ since $\sigma_2
  \subset F_2 $. So $\sigma_2^\perp \subset \mathfrak q_{F_2}^\perp
  \cap \liep= \mathfrak q_{F_2}'$. By dimension $\Dir(\sigma_2) $ is a
  maximal subalgebra in $\mathfrak q_{F_2}$ and $\sigma_2^\perp $ is a
  maximal subalgebra in $\mathfrak q_{F_2}'$.  On other hand from
  $g\cdot F_1 = F_2$ it follows that $g \cdot \Dir(\sigma_1) \subset
  \mathfrak q_{F_2}$ and $g \cdot \sigma_1^\perp \subset \mathfrak
  q_{F_2}$, and they are also maximal subalgebras in these spaces. By
  the Proposition \ref{deco-face} (b) and (c) there exist $k_1 \in
  K_1, k_2 \in K_2$ such that
  \begin{gather*}
    (k_1 g)\cdot \Dir(\sigma_1) = \Dir(\sigma_2)\\
    (k_2 g) \cdot \sigma_1^\perp = \sigma_2^\perp.
  \end{gather*}
  Since $x_0$ is fixed by the larger group $K^{\sigma_2^\perp}$ it
  follows that $k_1 g \sigma _1 = \sigma_2$.  Moreover $k_1k_2 = k_2
  k_1$ since $[\mathfrak h_{F_2}, \mathfrak h_{F_2}']=0$.  For the
  same reason $\mathfrak q_{F_2}'$ is fixed pointwise by $K_1$ and
  $\mathfrak q_{F_2}$ is fixed pointwise by $K_2$.  Set $k=k_1k_2$ and
  $w= kg$. Then $ k\in K^{\sigma_2 ^ \perp}$ and $w \in K$. We get
  \begin{gather*}
    w \cdot \Dir (\sigma_1) = \Dir(\sigma_2)\\
    w\cdot \sigma_1 ^\perp = \sigma_2 ^ \perp.
  \end{gather*}
  Thus $w \cdot \lia = \lia$, i.e. $w \in \mathcal N$.  Since $k\in
  K^{\sigma_2^\perp}$, $w \cdot F_1 = (k g) \cdot F_1 = k \cdot F_2
  =F_2$.  Since $\sigma_1 = (x_0 + \Dir(\sigma_1)) \cap F_1$ and
  similarly for $F_2$, we conclude that $w \sigma _1 = \sigma_2$.
  Finally we prove that $\tilde \Theta$ is surjective. Let $F\subset
  \c$ be a face. Then $F \subset \liep^{\lies}$ for some abelian
  subalgebra $\lies\in \liep$. Then there exists $k\in K$ such that $k
  \cdot \lia \subset \liep^{\lies}$. Therefore $k^{-1} \cdot F \subset
  \liep^{(k^{-1} \cdot \lies )}$ and $\lia \subset \liep^{(k^{-1}
    \cdot \lies )}$.  By Proposition \ref{faccia-faccetta}, $k \cdot
  F=K^{\sigma^\perp} \cdot \sigma$ where $\sigma=(k \cdot F) \cap
  \lia$ and so $\tilde \Upsilon$ is surjective.
\end{proof}
As an application of the above theorem and Proposition
\ref{faccia-faccetta2}, we get the following
result.
\begin{yuppi} \label{esposte-esposte} The faces of $E$ are exposed if
  and only if the faces of $P$ are exposed.
\end{yuppi}
\begin{proof}
  By the above Theorem, the map $\sigma \mapsto K^{\sigma^\perp}\cd
  \sigma$ induces a bijection between $\facesp/\mathcal N$ and
  $\faces/K$. Hence, keeping in mind that if $F_1 =kF_2$, then $F_1$
  is exposed if and only if $F_2$, the result follows from Proposition
  \ref{faccia-faccetta2}.
\end{proof}
\begin{remark}\label{compatible}
  We have proven Theorems \ref{main1} and \ref{esposte-esposte} under
  the assumption that $G$ is a connected real semisimple Lie group.
  From this it follows that both theorems hold true for any connected
  compatible subgroup of $U^\C$, since such a subgroup is real
  reductive in the sense of \cite[p. 446]{knapp-beyond} and thus it is
  the product of a semisimple connected subgroup and an abelian
  subgroup, see e.g. \cite[p. 453]{knapp-beyond}.
\end{remark}
\section{Convex hull of the gradient map image}\label{gradient-map}
Let $U$ be a compact connected Lie group and $U^\C$ its
complexification.  Let $(Z,\omega)$ be a K\"ahler manifold on which
$U^\C$ acts holomorphically. Assume that $U$ acts in a Hamiltonian
fashion with momentum map $\mu:Z \lra\liu^*$.  Let $G\subset U^\C$ be
a closed connected subgroup of $U^\C$ which is compatible with respect
to the Cartan decomposition of $U^\C$.  This means that $G$ is a
closed subgroup of $U^\C$ such that $G=K\exp (\liep)$, where $K=U\cap
G$ and $\liep=\lieg \cap i\liu$
\cite{heinzner-schwarz-stoetzel,heinzner-stoetzel-global}.  The
inclusion $i \liep \hookrightarrow \liu$ induces by restriction a
$K$-equivariant map $\mu_{i \liep}:Z \lra (i \liep)^*$.  Using a fixed
$U$-invariant scalar product $\scalo$ on $\liu$, we identify $\liu
\cong \liu^*$.  We also denote by $\scalo$ the scalar product on
$i\liu$ such that multiplication by $i$ be an isometry of $\liu$ onto
$i\liu$.  For $z \in Z$ let $\mup (z) \in \liep$ denote $-i$ times the
component of $\mu(z)$ in the direction of $i\liep$.  In other words we
require that $\sx \mup (z) , \beta \xs = - \sx \mu(z) , i\beta\xs$,
for any $\beta \in \liep$. Then we view $\mu_{i \liep}$ as a map
\begin{gather*}
  \mu_\liep : Z \ra \liep ,
\end{gather*}
which is called the $G$-\emph{gradient map} or \emph{restricted
  momentum map} associated to $\mu$.  For the rest of the paper we fix
a $G$-stable \emph{compact subset} $X \subset Z$ and we consider the
gradient map $\mup:X \lra \liep$ restricted on $X$.  We also set
\begin{gather*}
  \mupb:= \sx \mup, \beta \xs = \mu^{-i\beta}.
\end{gather*}
We will now study the convex hull of $\mup(X)$, that we denote by
$\cc$. Let $\lia\subset \liep$ be a maximal abelian subalgebra of
$\liep$ and let $\pi:\liep \lra \lia$ be the orthogonal projection
onto $\lia$. Then $\pi \circ \mup=:\mua$ is the gradient map
associated to $A=\exp(\lia)$. Let $Z^A$ be the set of fixed points of
$A$. We note that $\mua$ is locally constant on $Z^A$ since
$\mathrm{Ker}\, \mathrm{d} \mu_{\lia} (x)=(\lia\cdot x)^\perp$ (see
\cite{heinzner-stoetzel-global}).  Let $\mathfrak b$ a subspace of
$\lia$ and let $X^{\mathfrak b}=\{p\in X:\, \xi_X (p)=0$ for all $\xi
\in \mathfrak b\}$, where $\xi_X$ is the vector field induced by the
$A$ action on $X$.  Then the map $\mu_{\mathfrak b}:X^{\mathfrak b}
\lra \mathfrak b$, that is the composition of $\mup$ with the
orthogonal projection onto $\mathfrak b$, is locally constant
(\cite{heinzner-schuetzdeller}, Section $3$).  Since $X^{\mathfrak b}$
is compact, $\mu_{\mathfrak b} (X^{\mathfrak b})$ is a finite set. In
\cite{heinzner-schuetzdeller} it also shown that for any $y\in X^{(
  \mathfrak b )}:=\{p\in X:\, \lia_{p}=\mathfrak b \}$, where $\lia_p
:=\{\xi\in \lia:\, \xi_X (p)=0\}$, we have that $\mu_{\lia}(A\cdot y)
\subset \mu_{\lia} (y) + {\mathfrak b}^\perp$ is an open subset of the
affine space $\mu_{\lia} (y) + {\mathfrak b}^\perp$ (the orthogonal
complements are taken in $\lia$). Moreover $\mu_{\lia}(A \cdot y)$ is
a convex subset of $\mu_{\lia} (y) + {\mathfrak b}^\perp$ (see
\cite{heinzner-stotzel-2}) and therefore $\mu_{\lia}(\overline{A \cdot
  y})=\overline{\mu_{\lia}(A \cdot y)}$ is a convex body.

Let $P:=\widehat{\mua (X)}$. If $\beta\in \mu_{\lia}(X)$ is an
extremal point of $P$, and $y\in \mu_\lia^{-1}(\beta)$, then
$\mu_{\lia} (A \cdot y)$ is an open neighborhood of $\mu_{\lia} (y)$
in $\mu_\lia (y) + \mathfrak \lia_y^\perp$ and it is contained in
$\mu_\lia (X)\subset P$. Since $\mu_{\lia} (y)$ is an extremal point,
it follows that $\lia_y^\perp=\{0\}$ and so $y$ is a fixed point of
$A$.  Since $X$ is compact, the set $X^A$ has finitely many connected
components. Therefore $P$ has finitely many extremal points, i.e. it
is a polytope.  We have shown the following.
\begin{prop}\label{momentum-polytope}
  Let $X\subset Z$ be a $G$-invariant compact subset of $Z$.  Then the
  image $\mu_{\lia} (X^A)$ is a finite set $\{c_1, \ldots, c_p\}$ and
  $P=\widehat{\mu_{\lia} (X)}$ is the convex hull of $c_1,\ldots,c_p$.
\end{prop}
As a corollary we get the following result.
\begin{yuppi}
  Let $X\subset Z$ be a $G$-invariant compact subset of $Z$. Then
  every face of $E=\widehat{\mu_\liep (X)}$ is exposed.
\end{yuppi}
\begin{proof}
  Since
  \begin{gather*}
    \pi (\cc)= \widehat{\pi (\mup (X))}= \widehat{\mua(X)},
  \end{gather*}
  by Lemma \ref{gichev} (i) we conclude that $\cc \cap \lia = \pi(E)
  =P$ and by Proposition \ref{momentum-polytope}, Remark
  \ref{compatible} and Theorem \ref{main2} we get that every face of
  $\cc$ is exposed.
\end{proof}
We call $P$ the \emph{momentum polytope}. If $G=U^\C$ and $X$ is a
complex connected submanifold of $Z$, then $P=\mua(X)$ by the
Atiyah-Guillemin-Sternberg convexity theorem
\cite{atiyah-commuting,guillemin-sternberg-convexity-1}. The same
holds for $X$ an irreducible semi-algebraic subset of a Hodge manifold
$Z$ \cite{kostant-convexity,heinzner-schuetzdeller,bghc}.

Since any proper face $F$ of $ \cc$ is exposed, the set $C_F$ defined
in \eqref{def-CF} is a non-empty convex cone in $\liep$.  Set
\begin{gather*}
  K^F:=\{g\in K: g\cdot F = F\}.
\end{gather*}
By Proposition $5$ in \cite{biliotti-ghigi-heinzner-1-preprint} we
have $\CF^{K^F}:=\{\beta\in C_F:\, K^F \cd \beta=\beta\}\neq
\emptyset$.  This means that for a proper face $F$ one can find a
$K^F$--invariant vector $\beta$ such that $F_\beta(E) = F$.  For
$\beta \in \liep$, denote by $X^\beta$ the set of points of $X$ that
are fixed by $\exp (\R\beta)$.  If $\beta \in C_F$, let
\begin{gather*}
  \xmax:=\{x \in X: \mupb(x) = \max_X \mupb\}.
\end{gather*}
Since the function $\mup^\beta$ is $K^\beta$--invariant the set
$\xmax$ is $K^\beta$--invariant.  Moreover $\xmax$ is a union of
finitely many connected components of $X^\beta$ and $X^\beta$ is
$G^\beta$-stable.  Every connected component of $G^\beta$ meets
$K^\beta$. This implies that $G^\beta$ leaves $\xmax$ invariant. Next
we show that $\xmax$ does not depend on the choice of $\beta$ in $
C_F$.
\begin{lemma}\label{face-set}
  If $\beta \in C_F$, then $\xmax=\mup^{-1} (F)$. Moreover $F$ is the
  convex hull of $\mup(\xmax) $.
\end{lemma}
\begin{proof}
  Fix $x \in X$. Then 
  $\mup(x)\in F$ if and only if $ \sx \mup(x) , \beta \xs = \max_{v\in
    E} \sx v,\beta \xs$. Moreover $ \max_{v\in E} \sx v,\beta \xs =
  \max_{v\in \mup(X)} \sx v,\beta \xs = \max_{ X} \mupb$.  So $x \in
  \mup\meno(F)$ if and only if $x$ is a maximum of $\mup^\beta (x) $
  restricted to $X$.  This shows that $X_F^\beta=\mup^{-1} (F)$. The
  inclusion $\mup(X_F^\beta ) \subset F$ follows from the definition
  and therefore $\widehat {\mup(X_F^\beta ) } \subset F$.  By
  \cite[Lemma 3]{biliotti-ghigi-heinzner-1-preprint} $\ext F = \ext E
  \cap F$, so $\ext F \subset \mup(X) \cap F =\mup(X_F^\beta)$. It
  follows that $ F = \widehat{ \mup(X_F^\beta)}$.
\end{proof}
Motivated by the above Lemma we set $X_F:=\xmax$ where $\beta$ is any
vector in $C_F$. We also set
\[
Q^F =\{g\in G:\, g\cdot X_F =X_F \}.
\]
$Q^F$ is a closed Lie subgroup of $G$.

Given $\beta \in \liep$ define the following subgroups:
\begin{gather*}
  G^{\beta+} = \{g\in G:\, \lim_{t\mapsto -\infty }\exp (t\beta) g \exp (-t\beta)\ \mathrm{exists} \}, \\
  G^{\beta-} = \{g\in G:\, \lim_{t\mapsto +\infty }\exp (-t\beta) g
  \exp (t\beta)\ \mathrm{exists} \},
  \\
  R^{\beta+} = \{g\in G:\, \lim_{t\mapsto -\infty }\exp (t\beta) g \exp (-t\beta)=e \}, \\
  R^{\beta-} = \{g\in G:\, \lim_{t\mapsto +\infty }\exp (-t\beta) g
  \exp (t\beta)=e \}.
\end{gather*}
$G^{\beta+}$ (respectively $G^{\beta-}$) is a parabolic subgroup,
$R^{\beta+}$ (respectively $R^{\beta-}$) is its unipotent radical and
$G^\beta$ is a Levi factor. Therefore $G^{\beta + } = G^\beta \rtimes
R^{\beta+}$ (respectively $G^{\beta- } = G^\beta \rtimes R^{\beta-}$).

\begin{lemma}\label{parabolici-levi-compatto}
  $ Q^F \cap K = K^F$.
\end{lemma}
\begin{proof}
  If $g \in Q^F \cap K$, then $ g \cdot X_F = X_F$. Since $\mup$ is a
  $K$-invariant map, $g\cdot \mup(X_F) = \mup( X_F)$. Taking the
  convex hull of both sides and using Lemma \ref{face-set} we get that
  $g\cdot F= F$, thus $g \in K^F$. Conversely, if $g\in K^F$, the
  equivariance of $\mup$ yields $X_F = \mup\meno(F) =\mup\meno(g\cdot
  F) = g X_F$, thus $g\in Q^F$.
\end{proof}
We are now ready to prove the connection between the set of the faces
of $E$ and parabolic subgroups of $G$.
\begin{prop}\label{prop-face-parabolic}
  $Q^F$ is a parabolic subgroup of $G$. Moreover $Q^F=G^{\beta+}$ for
  every $\beta \in \CF^{K^F}$.
\end{prop}
\begin{proof}
  Observe that by definition $Q^F$ is a closed subgroup of $G$.  Let
  $\beta\in \CF^{K^F}$. Then $F=F_\beta (E)$ and, by definition of
  $K^F$, we get $K^F=K^\beta$. The set $X_F=\{x \in X: \mupb(x) =
  \max_X \mupb\}$ is $G^\beta$-stable.  Fix $p\in X_F$ and consider
  the orbit $G\cd p$, which is a smooth submanifold contained in
  $X$. By Proposition 2.5 in \cite{heinzner-schwarz-stoetzel} (see
  also Proposition 2.1 in \cite{biliotti-ghigi-heinzner-2}) we get
  that $\xi_X(x) = 0$ for any $\xi \in \lier^{\beta+}$ and for any
  $x\in X_F$.  Therefore $G^{\beta+} \cdot p \subset X_F$.  Hence
  $G^{\beta+} \subset Q^F$ and the Lie algebra $\mathfrak q^F$ of
  $Q^F$ is parabolic. On the other hand by Lemma
  \ref{parabolici-levi-compatto}, we have $\mathfrak q^F \cap
  \liek=\lieg^{\beta+}\cap \liek=\liek^\beta$ and so by Lemma 3.7
  \cite{biliotti-ghigi-heinzner-2} we conclude that $\mathfrak
  q^F=\lieg^{\beta+}$. Since $Q^F \subset N_G
  (\lieg^{\beta+})=G^{\beta+}$ we get $Q^F =G^{\beta+}$.
\end{proof}
\begin{remark}\label{rem-parabolic}
  If $\beta'\in \CF^{K^{F}}$, then $Q_F=G^{\beta'+}=G^{\beta+}$. By
  Lemma 2.8 in \cite{biliotti-ghigi-heinzner-2}, we have
  $[\beta,\beta']=0$, $G^\beta=G^{\beta'}$ and $R^{\beta
    +}=R^{\beta'+}$.
\end{remark}
Let $Q^{F-}=\Theta (Q^F)$, where $\Theta:G \lra G$ denotes the Cartan
involution.  The subgroup $Q^{F-}$ is parabolic and depends only on
$F$. The subgroup $L^F:=Q^F \cap Q^{F-}$ is a Levi factor of both
$Q^F$ and $Q^{F-}$. Let $\beta \in \CF^{K^F}$. Then $Q^F=G^{\beta+}$,
$L^F=G^\beta$ and we have the projection
\[
\pi^{\beta+}:G^{\beta+} \lra G^\beta, \qquad \pi^{\beta+} (g)=\lim_{t
  \mapsto + \infty } \exp (t\beta)h \exp (-t\beta),
\]
respectively
\[
\pi^{\beta+}:G^{\beta-} \lra G^\beta, \qquad \pi^{\beta-} (g)=\lim_{t
  \mapsto - \infty } \exp (t\beta)h \exp (-t\beta).
\]
\begin{lemma}\label{projection-parabolic}
  If $\beta\in \CF^{K^F}$, then the projections $\pi^{\beta+}$ and
  $\pi^{\beta-}$ depend only on $F$.
\end{lemma}
\begin{proof}
  Let $g\in G^{\beta +}$. We know that $g=hr$, where $h\in G^{\beta}$
  and $r\in R^{\beta +}$. Then
  \[
  \pi^{\beta+}(g)= \lim_{t \mapsto + \infty } \exp (t\beta)g \exp
  (-t\beta)=h\lim_{t \mapsto + \infty } \exp (t\beta)r \exp
  (-t\beta)=h.
  \]
  Since $G^\beta = G^{\beta'}$ and $R^{\beta +} = R^{\beta'+}$ the
  decomposition $g=hr$ is the same for both groups and
  $\pi^{\beta+}(g) = \pi^{\beta'+}(g)$.
  The same argument works for $\pi^{\beta-}$.
\end{proof}
Now assume that $X$ is a $G$-stable \emph{compact submanifold} of $Z$.

For $\beta \in C_F^{K_F}$ set $X^{\beta-}_F:=\{p\in X:\,
\lim_{t\mapsto +\infty} \exp(t\beta)\cdot p \in X_F \}$. Then the map
\begin{gather}
  p^{\beta-}: X^{\beta-}_F \lra X_F, \qquad
  p^{\beta-}(x)=\lim_{t\mapsto +\infty} \exp(t\beta)\cdot x
  \label{pinco}
\end{gather}
is well-defined, $G^\beta$-equivariant, surjective and its fibers are
$R^{\beta-}$-stable.

More generally one can consider $p^{\beta-}$ as a map from
$X^{\beta-}=\{y\in X:\ \lim_{t\mapsto +\infty} \exp (t\beta) \cdot x$
exists $\}$ to $ X^\beta $. In general however this map is not even
continuous \cite[Example 4.2]{heinzner-schwarz-stoetzel-arxiv}. To
ensure continuity and smoothness it is enough that the topological
Hilbert quotient $X^{\beta-}//G^\beta$ exists. Using the notation of
\cite{heinzner-schwarz-stoetzel-arxiv} and choosing $r=\max_{X}
\mup^\beta$, we have $X_F =\xmax = X^\beta_r$ and $X_r^{\beta-} =
X_F^{\beta -}$. Thus Prop. 4.4 of
\cite{heinzner-schwarz-stoetzel-arxiv} applies and yields that
$X_F^{\beta -}$ is an open $G^{\beta-} $--stable subset of $X$ and
that \eqref{pinco} is smooth deformation retraction onto $X_F$.  Using
$\pi^{\beta-}$ one defines an action of $Q^{F-}=G^{\beta-}$ on $X_F$
by setting $g \cd x=\pi^{\beta-}(q)\cd x$.  This just depends on
$F$. With respect to this action the map $p^{\beta-}$ becomes
$Q^{F-}$--equivariant.
\begin{lemma}\label{retraction-parabolic}
  The set $X^{\beta-}_F$ and the map $p^{\beta-}$
  do not depend on the choice of $\beta\in \CF^{K^F}$.
\end{lemma}
\begin{proof}
  Set $\Gamma=\exp (\R\beta)$.  If $p\in X_F$ by the Slice Theorem
  \cite[Thm. 3.1]{heinzner-schwarz-stoetzel} there are open
  neighborhoods $S_p\subset T_p X$ and $\Omega_p\subset X$ and a
  $\Gamma$-equivariant diffeomorphism $\Psi_p:S_p \lra \Omega_p$, such
  that $0\in S_p$, $p\in \Omega_p$, $\Psi_p(0)=p$. Since $p$ is a
  maximum of $\mup^\beta$ restricted to $X$, the following orthogonal
  splitting $T_p X=V_0 \oplus V_{-}$ with respect to the Hessian of
  $\mup^\beta$ holds.  Here $V_0$ denotes the kernel of the Hessian of
  $\mup^\beta$ and $V_{-}$ denotes the sum of eigenspaces of the
  Hessian of $\mup^\beta$ corresponding to negative eigenvalues.  We
  also point out that $V_0=T_p X_F$ and $S_p = \{ x_0 + x_- : x_0 \in
  S_p \cap V_0 , x_-\in V_-\}$, see \cite{heinzner-stoetzel-global}.
  It follows that $\Omega_p \subset X_F^{\beta-}$.  Set $\Omega:=
  \bigcup_{p\in X_F} \Omega_p$. By what we just proved, $\Omega
  \subset X_F^{\beta-}$. On the other hand $\Omega $ is an open
  $\Gamma$-invariant neighbourhood of $X_F$, so $X_F^{\beta-} \subset
  \Omega$.  So $X_F^{\beta-} = \Omega$. If $\beta'$ is another vector
  of $C_F^{K^F}$, set $B=\exp (\R \beta \oplus \R \beta')$. This is a
  compatible abelian subgroup and $X_F \subset X^B$. So we may choose
  the open subsets $\Omega_p$ above to be $B$-stable. Therefore we get
  $X^{\beta'-}=\Omega$ as well. This proves that $X_F^{\beta-} =
  X_F^{\beta'-}$.

  Next we show that $p^{\beta-} = p^{\beta'-}$.  First observe that
  $p^{\beta-} (y) = p^{\beta'-}(y)$ if $y\in \Omega$. Indeed if $y\in
  \Omega_p$ we can study the limit using the diffeomorphism $\Psi_p:
  S_p \ra \Omega_p$.  The decomposition $T_pX =V_0 \oplus V_-$ is the
  same for $\beta$ and $\beta'$ since they commute and attain their
  maxima on $X_F$.  Therefore if $x=\Psi_p\meno(y) = x_0 + x_-$, then
  \begin{gather}
    \label{formula}
    p^{\beta-} (y) = \Psi_p( x_0)= p^{\beta'-}(y).
  \end{gather}
  If $p\in X^{\beta-}_F$ and $q=\lim_{t \mapsto +\infty } \exp (t
  \beta )\cd p \in X_F$, there is $t_1\in \R$, such that $ \exp (t
  \beta )\cd p \in \Omega$.
  Therefore
  \[
  \begin{split}
    \lim_{t\mapsto +\infty} \exp (t \beta' )\cd p &= \lim_{t\mapsto +\infty} \exp(t\beta') (\exp(t_1 \beta') \cd p) \\
    &= \lim_{t\mapsto +\infty} \exp(t\beta) (\exp(t_1 \beta') \cd p) \ (\mathrm{by}\ \ref{formula})\\
    &=\exp(t_1 \beta') (\lim_{t\mapsto +\infty} \exp(t\beta) \cd p) \\
    &=\lim_{t\mapsto +\infty} \exp(t\beta) \cd p.
  \end{split}
  \]
\end{proof}
By the above Lemma if $F$ is a face and $\beta\in C_F^{K^F}$, we can
set $X_F^{-}:=X^{\beta-}_F$ and $p^{F-}:=p^{\beta-} : X_F^{-} \lra
X_F$.
\begin{teo}\label{urg}
  For any face $F\subset \cc$, the set $X_F$ is closed and
  $L^{F}$-stable, $X_F^{-}$ is an open $Q^{F-}$--stable neighborhood
  of $X_F$ in $X$ and the map $p^{F-}$ is a smooth
  $Q^{F-}$--equivariant deformation retraction of $X_F^{-}$ onto $
  X_F$.
\end{teo}
\def\cprime{$'$}

\end{document}